\documentclass[11pt,reqno]{amsart}
\usepackage{xifthen}
\usepackage{subfig}
\usepackage{pkgfile}

\usepackage{amssymb,amsfonts,amsmath,amsthm,amscd,dsfont,mathrsfs}
\usepackage{graphicx,float,psfrag,epsfig}
\usepackage{wrapfig}

\DeclareMathAlphabet{\mathpzc}{OT1}{pzc}{m}{it}
\DeclareMathOperator{\sech}{sech}

\newtheorem{propo}{Proposition}[section]
\newtheorem{lemma}[propo]{Lemma}

\newtheorem{hyp}{Hypothesis}

\newcommand{\reals}{{\mathds R}}

\newcommand{\eqnsection}{\renewcommand{\theequation}{\thesection.\arabic{equation}}
      \makeatletter \csname @addtoreset\endcsname{equation}{section}\makeatother}

\def\l|{\left|\left|}
\def\r|{\right|\right|}
\def\R{\mathbb R}
\def\P{\mathbb P}

\def\E{\mathbb E}

\def\ve{\varepsilon}
\def\diag{\textrm {diag}} 

\def\de{{\rm d}}
\def\reals{{\mathds R}}
\def\us{\underline{\sigma}}

\def\R{\mbox{\tiny\rm R}}

 \def\us{{\underline\sigma}}

\title{High Temperature Asymptotics of Orthogonal Mean-Field Spin Glasses}

\author{Bhaswar B. Bhattacharya}
\address{Department of Statistics, Stanford University, California, USA,  \tt{bhaswar@stanford.edu}}

\author{Subhabrata Sen} 
\address{Department of Statistics, Stanford University, California, \tt{ssen90@stanford.edu}} 

\date{\today}

\begin{document}

\begin{abstract}
 %
We evaluate the high temperature limit of the free energy of spin glasses on the hypercube with Hamiltonian $H_N(\us) = \us^T J \us$, where the coupling matrix $J$ is drawn from certain symmetric orthogonally invariant ensembles. Our derivation relates the \textit{annealed free energy} of these models to a spherical integral, and expresses the limit of the free energy in terms of the limiting spectral measure of the coupling matrix $J$. As an application, we derive the limiting free energy of the Random Orthogonal Model (ROM) at high temperatures, which confirms non-rigorous calculations of Marinari et al. \cite{marinari_parisi}. Our methods also apply to other well-known models of disordered systems, including the SK and  Gaussian Hopfield models. 
\end{abstract}


\maketitle

\section{Introduction}
\label{sec:intro}


Consider a (random) function on the hypercube $H_N: S_N = \{-1, +1\}^N \to \reals$ defined as  
\begin{align}
H_N(\us)= \us^{T} J \us
\label{orthogonal} 
\end{align}
with coupling matrix $J=ODO^{T}$, where $O$ is Haar distributed over the orthogonal group $O(N)$ and $D=\diag (d_1, \cdots, d_N)$ is a diagonal matrix independent of $O$.  This defines a probability distribution over $S_N$ as follows: for $\tau \in S_N$ and $\beta \geq 0$,
\begin{equation}
\P(\vec\sigma= \tau)=\frac{1}{2^N}\cdot\frac{e^{\beta H_N(\tau)}}{Z_N(\beta, O, D)},
\label{exp}
\end{equation}
where the {\it partition function} $Z_N(\beta, O, D) = \frac{1}{2^N}\sum_{\us \in  S_N} \exp(\beta H_N(\us))$. These distributions arise frequently in the analysis of disordered systems in statistical physics. In this context,  $H_N(\us)$ describes the energy of the configuration $\us$, and is usually referred to as the Hamiltonian of the system. The parameter $\beta$ denotes the inverse temperature, so the {\it high temperature} regime corresponds to small values of $\beta$. We seek to evaluate the large $N$ limit of the {\it free energy} 
\begin{equation}
\Phi_N(\beta, O, D)=\frac{1}{N}\log Z_N(\beta, O, D)
\label{phi}
\end{equation}
in these models. 

Models of the form (\ref{exp})  will be referred to as {\it orthogonal mean-field spin glasses}--- they include many well-known physical models of disordered systems:

\begin{enumerate}[(a)]

\item {\it Sherrington-Kirkpatrick (SK) Model}: In the SK model of spin glasses the coupling matrix $J=\frac{1}{\sqrt N}W$, where $W$ is a symmetric matrix drawn from the {\it{Gaussian Orthogonal Ensemble}}.  
It is well known that $W=O D O^{T}$,  where $O\sim O(N)$ is Haar distributed and $D=\diag (d_1, d_2, \cdots, d_N)$ is a diagonal matrix independent of $O$, such that the empirical measure $\frac{1}{N}\sum_{i=1}^N \delta_{d_i}$ converges to the semi-circle law \cite{agz}. 
The limit of the free energy for all temperatures was conjectured by Parisi using deep ideas of replica symmetry breaking, and was rigorously established by Talagrand \cite{talagrand} (refer to \cite{panchenko} for an introduction to this subject).  Carmona and Hu \cite{carmona_hu} (see also Chatterjee \cite{chatterjee}) proved that the Parisi formula continues to hold even if the entries of the coupling matrix $J=((J_{ij}))$ are independent mean zero random variables, subject to some conditions on the higher moments. 

\item {\it Random Orthogonal Model (ROM)}: Marinari et al. \cite{marinari_parisi} introduced the ROM to model a deterministic system which exhibits glassy behavior. In this model the coupling matrix $J= ODO^T$, where $D= \diag (d_1, \cdots, d_N)$ is a deterministic sequence of $\{ \pm 1\}$  such that the empirical measure 
\begin{equation}
\mu_N(D)=\frac{1}{N}\sum_{i=1}^N \delta_{d_i}\dto  p\delta_{1} + (1-p) \delta_{-1},
\label{rom}
\end{equation}
for some $p \in (0, 1)$. The case $p=1/2$ has received a lot of attention in the physics literature (see \cite{bernasconi,energy_landscape,marinari_parisi_autocorrelation} and the references therein). The limiting free energy of this model is not known rigorously even in the high temperature regime. The coupling matrix $J$ has dependent entries and non-rigorous calculations based on the replica method 
%
predict different behavior compared to the SK model \cite{cherrier,marinari_parisi_autocorrelation,marinari_parisi}. This suggests that comparison/universality techniques like \cite{carmona_hu,chatterjee} cannot be directly used to compute the free energy.

\item {\it Gaussian Hopfield Model}: Cherrier et al. \cite{cherrier} considered the Gaussian Hopfield Model where the coupling matrix $J= \frac{1}{p} XX^T$, where $X=((X_{ij}))$ is a $ N \times p$ matrix with i.i.d. $\cN(0, 1)$. The coupling matrix of the usual Hopfield model has the same structure, but the matrix $X$ consists of i.i.d. Rademacher $\{\pm 1\}$ random variables.
Bovier et al. \cite{bovier} studied the Gaussian Hopfield model with 2-patterns and this ``simple" case already shows highly complicated behavior. It is generally believed that a Hopfield model with $p$ parameters where $p \sim \lambda N$ is significantly more complicated compared to the one with a finite number of patterns. 
\end{enumerate}


This paper gives a general method for computing the limit of the free energy in orthogonal mean-field spin glass models at sufficiently high temperatures (see Theorem \ref{main}). Exploiting a connection with spherical integrals \cite{hchandra,guionnet_maida} and using techniques from large deviations and random matrix theory, we rigorously justify certain heuristics employed in the traditional analyses of these systems. In particular, we derive: 

\begin{itemize}
\item[1.] the limiting free energy of the SK model in the entire high temperature phase (Corollary \ref{sk}), re-deriving the classical result of Aizenman et. al.  \cite{alr},

\item[2.] the limiting free energy of ROM for $\beta$ sufficiently small (Corollary \ref{cor:rom}), which verifies predictions of Marinari et al. \cite{marinari_parisi}, and

\item[3.] the limiting free energy of the Gaussian Hopfield model with $p/N \to \lambda \in (1, \infty)$ for high temperatures, confirming non-rigorous calculations of Cherrier et al. \cite{cherrier}. We remark that our techniques should also apply to the case $ \lambda \in (0,1)$ but we restrict ourselves to the first case for clarity. 

\end{itemize}

\subsection{Main Results} To state our main results we need to introduce some notations. The Haar measure on the orthogonal group $O(N)$ will be denoted by $\mathrm  d O$, and the expectation of a function $f$ will be denoted by $\E_0f(O):=\int_{O}f(O)\mathrm  d O$.

For any probability measure $\mu$, denote by $\textrm{supp}$ the support of $\mu$. We will always consider probability measures with bounded support so that $\textrm{supp}(\mu) \subseteq [\lambda_{\min}, \lambda_{\max}]$. To describe our results we need to introduce the  Hilbert transform and the $R$-transform of a probability measure with bounded support: 

\begin{defn}
The Hilbert transform $H_{\mu}$ of a measure $\mu$ is $H_{\mu}: \mathbb{R}\backslash \textrm{supp}(\mu) \to \mathbb{R}$ 
\begin{align}
z &\mapsto \int \frac{1}{z-\lambda} \de\mu(\lambda). 
\label{hilbert_transform}
\end{align}
It is easy to show that $H_{\mu}$ is a bijective map from $\reals\backslash \textrm{supp}(\mu)$ to $(H_{\min}, H_{\max})\backslash\{0\}$ (see \cite{guionnet_maida}), where
\begin{align}
\label{eq:hmax}
H_{\max} = \lim_{ z \downarrow \lambda_{\max}} H_{\mu}(z) \,\,\,\,\,\,\,\,\, H_{\min}= \lim_{ z \uparrow \lambda_{\min}} H_{\mu}(z).  
\end{align}

Thus, setting $x_{\min} = \lambda_{\min} - 1/H_{\min}$, $x_{\max}= \lambda_{\max}- 1/H_{\max}$, and $m= \int \lambda \de \mu_{\lambda}$,  for $z \in H_{\mu}(\mathbb{R}\backslash {\textrm{supp}(\mu)})$, define the $R$-transform $R_{\mu}: (H_{\min}, H_{\max})\backslash\{0\} \to (x_{\min}, x_{\max}) \backslash\{m\} $ as 
\begin{align}
H_{\mu}\Big( R_{\mu}(z) + \frac{1}{z} \Big) = z. \label{R_transform} 
\end{align}
It is easy to see that $R_{\mu}$ is bijective and we denote its inverse by $Q_{\mu}$. Let 
\begin{equation}
I_{\mu}(\beta) = \frac{1}{2} \int_{0}^{2\beta} R_{\mu}(v) \de v.
\label{Ibeta}
\end{equation} 
Finally, for any $\beta >0$, define 
\begin{align}\label{interval}
U_L = 2\beta ( 1 - \tanh \beta( x_{\max} - x_{\min}) ),   \quad U_R = 2 \beta ( 1 + \tanh \beta (x_{\max} - x_{\min}) ).
\end{align}
\end{defn}


We will restrict ourselves to models where the sequence of random empirical measures $\mu_N(D)=\frac{1}{N}\sum_{i=1}^N \delta_{d_i}$ corresponding to the matrix $D$ in \eqref{exp} satisfy certain ``rigidity" properties. This allows us to neglect the fluctuations of the spectrum in the calculation of the free energy limit. We impose the following property on the law of the matrix $D$. 

\begin{hyp}
\label{hyp:dist} 
Let $D=\diag(d_1, d_2, \ldots, d_N)$ be a (random) diagonal matrix with  empirical measure $\mu_N(D)=\frac{1}{N}\sum_{i=1}^N \delta_{d_i}$. Assume that 
\begin{enumerate}
\item  there exists a sequence of numbers $M_N = o(\sqrt{N})$ such that 
$$\lim_{N\rightarrow \infty}\P( \|D \|_{\infty} > M_N) =0,$$
where $\| D \|_{\infty} = \max_{1 \leq i \leq N} |d_i|$;
\item there exists a deterministic measure $\nu_N$ supported on $N$ points in $\reals$ such that for any $c>0$, 
\begin{equation}
\lim_{N\rightarrow \infty}\P\left(W_2 (\mu_N(D), \nu_N) > \frac{c}{\sqrt{N}}\right) \to 0
\label{W2}
\end{equation}
where $W_2(\cdot, \cdot)$ is the 2-Wasserstein distance between two probability measures. 
\end{enumerate}
\end{hyp}

In most of our applications, it suffices to take $\nu_N=\frac{1}{N}\sum_{i=1}^N\delta_{\E(d_i)}$. It can be easily checked that all our results continue to hold with any sequence of probability measures $\nu_N$ satisfying Hypothesis \ref{hyp:dist}. However, we state our results with $\nu_N= \sum_{i=1}^{N} \delta_{\E(d_i)}$ for clarity.      We define,  for any deterministic diagonal matrix $\Lambda = \textrm{diag}(\lambda_1, \cdots, \lambda_N)$,
\begin{align}
\Gamma_N(\beta, \Lambda) = \frac{1}{N} \E_0( \log Z_N(\beta, O, \Lambda)).
\label{gamma}
\end{align} 
Note that due to the invariance of the Haar measure on $O(N)$, $\Gamma_N(\beta, \Lambda)$ is only a function of the empirical distribution $\mu_N(\Lambda)=\frac{1}{N}\sum_{i=1}^N \delta_{\lambda_i}$. The following proposition establishes that we may neglect the fluctuations of the spectrum for the calculation of the free energy. 

\begin{ppn}
\label{ppn:first_step}
Consider an orthogonal mean field spin glass model (\ref{exp}) with a (random) diagonal matrix $D=\diag(d_1, d_2, \ldots, d_N)$. If the sequence of measures $\mu_N(D)=\frac{1}{N}\sum_{i=1}^N \delta_{d_i}$ satisfies  Hypothesis \ref{hyp:dist}, then
\begin{align}
|\Phi_N(\beta, O, D) - \Gamma_N(\beta , \E(D) ) |\pto 0.
\end{align}
\end{ppn}

The proof of Proposition \ref{ppn:first_step} is outlined in Section \ref{pfppn}. Given this result, to compute the limit of the free energy $\lim_{N\rightarrow \infty}\Phi_N(\beta, O, D)$ it suffices to compute the limit of $\Gamma_N(\beta, \E(D))$. 

A crucial ingredient in the analysis of the asymptotics of $\Gamma_N(\beta, \Lambda)$  is a connection with a spherical integral. 
Guionnet and Maida \cite{guionnet_maida} derived the asymptotics of these integrals in terms of the $R$-transform of the limit $\mu$ of the empirical measure $\mu_N(\Lambda)=\frac{1}{N}\sum_{i=1}^N \delta_{\lambda_i}$ (refer to Section \ref{sec:spherical} for details). They assume the following conditions on the measure $\mu_N(\Lambda)$:

\begin{hyp}
\label{hyp:exp_dist}
For a deterministic diagonal matrix $\Lambda=\diag(\lambda_1, \lambda_2, \ldots, \lambda_N)$, denote by $\mu_N(\Lambda)=\frac{1}{N}\sum_{i=1}^N \delta_{\lambda_i}$ the empirical measure of $\Lambda$. Assume that
\begin{enumerate}
\item the sequence of measures $\{\mu_N(\Lambda)\}_{N\geq 1}$ converges weakly to a compactly supported measure $\mu$, and

\item $\lambda_{\min}(\Lambda) := \min_{1\leq i \leq N} \lambda_i$ and $\lambda_{\max}(\Lambda) := \max_{1\leq i\leq N} \lambda_i$ converge to $\lambda_{\min}$ and $\lambda_{\max}$ which are finite. 
\end{enumerate}
\end{hyp}

We will also assume Hypothesis \ref{hyp:exp_dist} to determine the limit of the partition function $Z_N(\beta, O, D)$. We have the following general result for the limiting free energy at high temperature. 

 \begin{thm}
\label{main}
Consider an orthogonal mean field spin glass model (\ref{exp}) with a (random) diagonal matrix $D=\diag(d_1, d_2, \ldots, d_N)$. Assume that 
\begin{enumerate}[(a)]

\item the sequence of (random) measures $\mu_N(D)=\frac{1}{N}\sum_{i=1}^N \delta_{d_i}$ satisfies Hypothesis \ref{hyp:dist}, and 

\item the sequence of deterministic measures $\mu_N(\E(D))=\frac{1}{N}\sum_{i=1}^N \delta_{\E(d_i)}\dto \mu$ and satisfies Hypothesis \ref{hyp:exp_dist}. 


\item $\limsup_{\beta \to 0}  \beta^2 \sup_{\beta_0 \in [ U_L, U_R] } | R'(\beta_0)| < \frac{1}{4}$, where $U_L$ and $U_R$ are as defined in~\eqref{interval}. 
\end{enumerate}

\noindent Then for $\beta$ sufficiently small (depending on $\mu$),
\begin{align}
\Phi_N(\beta, O, D) \pto I_{\mu}(\beta),
\label{phiNlimit}
\end{align}
with $I_{\mu}$  defined in (\ref{Ibeta}).
\end{thm}

As a consequence of the above theorem, we obtain the limiting free energy for many well-known models of disordered systems. Most importantly, we derive the limiting free energy of ROM (\ref{rom}) for $\beta$ sufficiently small (Corollary \ref{cor:rom}), which matches the predictions of Marinari et al. \cite{marinari_parisi} obtained by non-rigorous methods. The limiting free energy for the case $p=1/2$ is given in the following corollary. Refer to Proposition \ref{ppn:rom} for the expression for any $p\in (0, 1)$.

\begin{cor}
For the random orthogonal model (ROM) with $p=1/2$, for $\beta$ sufficiently small, 
\begin{align}
\frac{1}{N} \log Z_N(\beta, O, D) \pto \frac{1}{4}\left(\sqrt{16 \beta ^2+1}+\log \left(\frac{\sqrt{16 \beta ^2+1}-1}{8 \beta ^2}\right)-1\right).
\label{betag}
\end{align}
\label{cor:rom}
\end{cor}

Using Theorem \ref{main} we can also obtain the limiting free energy of the SK model in the entire high temperature phase (Corollary \ref{sk}), re-deriving the classical result of Aizenman et al.~\cite{alr}. Our calculations also give the limiting free energy for the Gaussian Hopfield model at high temperatures, verifying non-rigorous calculations of Cherrier et al. \cite{cherrier}.


\subsection{Proof Outline and Connections to Spherical Integrals} 
\label{sec:spherical}

Spherical integrals over the orthogonal group $O(N)$ (also known as Harish Chandra-Itzykson-Zuber (HCIZ) integrals \cite{hchandra})  are integrals of the form 
\begin{align}
\int_{O(N)} \exp( N \tr (OD_NO^TE_N)) \mathrm d O,
\label{sphericalint}
\end{align}
where $D_N$ and $E_N$ are $N\times N$ diagonal matrices. HCIZ integrals have been studied due to their connection to matrix models and the enumeration of planar maps (refer \cite{guionnet_zeitouni} and the references therein). Asymptotics of spherical integrals was studied  by Guionnet and Maida \cite{guionnet_maida} in the regime where the rank of $D_N$ is small compared to $N$.  An alternative simpler proof was provided in \cite{sniady}. 

To see the connection of such integrals to mean-field orthogonal spin glass models consider the  {\it annealed free energy} of the model (\ref{exp}): $\phi_N(\beta, \Lambda)=\frac{1}{N}\log \E_0 Z_N(\beta, O, \Lambda)$, where $\Lambda=\diag(\lambda_1, \lambda_2, \ldots, \lambda_N)$ is a deterministic diagonal matrix. Note that
\begin{equation}
Z_N(\beta, O, \Lambda)=\frac{1}{2^N}\sum_{\vec \sigma\in S_N}\exp(\beta \vec \sigma^T O\Lambda O^T \vec \sigma)=\frac{1}{2^N}\sum_{\vec \sigma \in S_N} \exp \left( N\beta \tr\left\{O\Lambda O^{T} \left(\frac{\us\,\us^T}{N} \right) \right\}\right).
\label{eq:zn_I}
\end{equation}
By the spectral decomposition $\frac{\us\,\us^{T}}{N}=P_\us E_{11}P_\us$ where $E_{11}=\diag(1, 0, \ldots, 0)$. Using (\ref{eq:zn_I})  and the invariance of the Haar distribution, \begin{eqnarray}
\E_0(Z_N(\beta, O, \Lambda))= \E_0\exp\left(N\beta \tr\left\{O\Lambda O^{T}E_{11}\right\}\right).
\label{eq:zn_II}
\end{eqnarray}
This is exactly of the form (\ref{sphericalint}) with $D_N=\Lambda$ and $E_N=E_{11}=\diag(1, 0, \ldots, 0)$.
Therefore, the annealed free energy $\phi_N(\beta, \Lambda)$ for any deterministic diagonal matrix $\Lambda$, is given by a spherical integral.  The limit of $\phi_N(\beta, \Lambda)$ was derived by Guionnet and Maida \cite{guionnet_maida}, when Hypothesis \ref{hyp:exp_dist} holds:

\begin{thm}[Guionnet and Maida \cite{guionnet_maida}]
\label{thm:annealed}
Consider an orthogonal mean field spin glass model (\ref{exp}) with a deterministic diagonal matrix $\Lambda=\diag(\lambda_1, \lambda_2, \ldots, \lambda_N)$. If the sequence of empirical measures $\mu_N(\Lambda)=\frac{1}{N}\sum_{i=1}^N \delta_{\lambda_i}\dto \mu$ and Hypothesis \ref{hyp:exp_dist} holds, then for $\beta$ sufficiently small (depending on $\mu$) 
\begin{align}
 \lim_{N\rightarrow \infty}\phi_N(\beta, \Lambda) =  I_{\mu}(\beta).
\end{align}
\end{thm}

The proof of Theorem \ref{main} proceeds as follows: when $D$ is random in (\ref{exp}), then under Hypothesis \ref{hyp:dist} we can replace the random matrix $D$ by the deterministic matrix $\E(D)$. Theorem \ref{main} then involves computing the limit of the annealed free energy $\phi_N(\beta, \E(D))$ using the above theorem, and the corresponding second moment. This together with results about concentration of measure gives the desired result.

\begin{remark}
When $D$ is random, another natural approach is to compute the {\it total annealed free-energy} $\phi_N^{\textrm{ann}}(\beta) = \frac{1}{N} \log \E(Z_N(\beta, O , D))$,  where the expectation is respect to the joint distribution of $(O,D)$. From (\ref{eq:zn_I}) it is easy to see that 
\begin{align}
\label{eq:representation}
\phi_N^{\textrm{ann}}(\beta) = \frac{1}{N} \log \E\left( N \beta \frac{\sum_{i=1}^{N} d_i X_i^2}{\sum_{i=1}^{N} X_i^2} \right), 
\end{align}
where the $X_i$ are i.i.d. $\dN(0,1)$ random variables.

It is expected that for $\beta$ sufficiently small, this also gives the correct limit for the free energy. To this end, consider the random measure  $\nu_N = \sum_{i=1}^{N} \frac{X_i^2}{\sum_{i=1}^{N} X_i^2} \delta_{d_i}$, i.e., $\nu_N$ is a random  discrete measure which assigns random weights $\frac{X_i^2}{\sum_{i}X_i^2}$ to the random positions $d_i$. Gamboa and Rouault \cite{gamboa_rouault} derived a large deviation principle for the random measure $\nu_N$, under certain technical assumptions on the sequence  $\mu_N(D)=\frac{1}{N}\sum_{i=1}^N \delta_{d_i}$.
%
We believe that under these assumptions a second moment argument can be done to derive the high temperature limit of the free energy $\Phi_N(\beta, O, D)$. However, this requires the full large deviation principle for the sequence $\{\mu_N(D)\}_{N\geq 1}$. On the other hand, we only need control on the tails of $\mu_N(D)$ in terms of the 2-Wasserstein distance, which is generally much easier to verify.
\end{remark}

\subsection{Organization} The rest of the paper is organized as follows: The proof of Corollary \ref{cor:rom} and the application of Theorem \ref{main} to various other examples are given in Section \ref{sec:examples}. The proofs of Proposition \ref{ppn:first_step} and Theorem \ref{main} are given in Section \ref{pfppn} and Section \ref{pfmain}, respectively.

\section{Examples}
\label{sec:examples}
In this section, we apply Theorem \ref{main} to evaluate the limit of the free energy in various orthogonal mean-field spin glass models. 

\subsection{The SK Model}

Recall the definition of the SK-model introduced in Section \ref{sec:intro}. In this case, the coupling matrix $J = W/\sqrt{N}$, where $W$ is a GOE matrix of order $N$. Thus $J = ODO^T$, where $O$ is Haar distributed and independent of $D$. It is a classical result in random matrix theory that $\mu_N(D)=\frac{1}{N}\sum_{i=1}^N \mu_N(D)$ converges almost surely to the Wigner semicircle law \cite{agz} 
\begin{align}
\rho(x) = \frac{\sqrt{4- x^2} }{2\pi} \cdot \pmb 1\{x\in [-2, 2 ]\}.
\label{semicircle} 
\end{align}
Further, the edge of the empirical distribution converges to the edge of the semicircle law. 

An application of Theorem \ref{main} yields the following corollary about the high temperature limit of the free-energy. It is well known that the SK model has a phase transition at $\beta=1/2$. Our approach covers the whole high temperature region of the SK model, thus re-deriving the classical result of Aizenman et. al. \cite{alr}.

\begin{cor}
\label{sk}
For the SK model with $\beta <1/2$, $\lim_{N\rightarrow \infty}\frac{1}{N} \log Z_N(\beta, O, D) \pto \beta^2.$ 
\end{cor}

\begin{proof}In this case $\lambda_{\min} = -2$, $\lambda_{\max}=2$. Using the density of the semi-circle law (\ref{semicircle}), the Hilbert transform can be easily computed to be $H_{\rho}(z) = \frac{1}{2}(z- \sqrt{z^2 - 4})$ for $z\in \reals\setminus [-2, 2]$. This implies $H_{\max}= 1$,  $H_{\min}= -1$, and $x_{\max}= 1$, $x_{\min}= -1$. Thus, using Definition \ref{R_transform},  $R_{\rho}(z) = z$ on $(-1,1)\backslash\{0\}$, $I_\rho(z)=z^2$ , and condition (c) in Theorem \ref{main} holds trivially. This gives the desired conclusion subject to the verification of the other conditions of Theorem \ref{main}. 

It is well known that the measure $\mu_N(\E(D)):=\frac{1}{N}\sum_{i=1}^N \delta_{\E(d_i)}$ satisfies Hypothesis \ref{hyp:exp_dist} \cite{agz}.
Further, by \cite[Corollary 4]{dallaporta1}  there exists $C>0$ such that
\begin{align}
\E \{W_2( \mu_N(D), \mu_N(\E(D) ))\} \leq C \frac{\sqrt{\log N}}{N},
\end{align} 
Hypothesis \ref{hyp:dist} then follows using Markov's inequality.

To see that the second moment method employed in our proof works up to $\beta<1/2$, see Remark \ref{rem:sk}. 
\end{proof}

\subsection{The Random Orthogonal Model} 
\label{sec:rom}

In the random orthogonal model (ROM) introduced in Section \ref{sec:intro} the coupling matrix $J= ODO^T$, where $D= \diag (d_1, \cdots, d_N)$ is a deterministic sequence of $\{ \pm 1\}$ such that the empirical measure $\mu_N(D)$ converges weakly to $\mu_p:=p\delta_{1} + (1-p) \delta_{-1}$.

\begin{ppn}
\label{ppn:rom} For the random orthogonal model (ROM), there exists a $\beta_{m}>0$ such that for $\beta<\beta_m$, 
\begin{align}
\frac{1}{N} \log Z_N(\beta, O, D) \pto \frac{1}{2}\int_0^{2\beta} \frac{ \sqrt{1+ 4z (m+z)} -1}{2z} \mathrm dz.
\label{rom1}
\end{align}
where $m = 2p -1$. 
\end{ppn}

\begin{proof}In this case, the diagonal matrix $D$ is deterministic. Thus, Hypothesis \ref{hyp:dist} holds trivially. Also, since the limiting measure $\mu_p$ is supported on two points,  Hypotheses \ref{hyp:exp_dist} is satisfied. 

In this case, $\lambda_{\max}=1$ and $\lambda_{\min}=-1$. Moreover, by direct calculations $H_{\mu_p}(z) = \frac{z+ m}{z^2 -1}$ on $(-\infty, -1) \cup (1, \infty)$. Thus, $H_{\max}=\infty$ and $H_{\min}= - \infty$ which implies that $$R_{\mu_p}(z) = \frac{1}{2z}\left(\sqrt{1+ 4z (m+z)} -1\right)$$ is bijective from $\reals\backslash\{0\}$ to $(-1,1)\backslash\{m\}$.

To verify condition (c) in Theorem~\ref{main}, recall the definition of $U_L$ and $U_R$ from~\eqref{interval}. For $\beta_0\in [U_L, U_R]$, there exists $\varepsilon(\beta_0)$ such that $|\varepsilon(\beta_0)|\leq |\tanh\beta(x_{\max}-x_{\min})|$ and $\beta_0=2\beta(1+\varepsilon(\beta_0))$. Thus, 
\begin{align}
R'_{\mu_p}(\beta_0) = \frac{16 \beta^2 (1+\varepsilon(\beta_0))^2 -1 + \sqrt{1 + 4 \beta_0(m +\beta_0)} }{16 \beta^2 (1+\varepsilon(\beta_0))^2 \sqrt{1 + 4 \beta_0(m +\beta_0)}}. 
\end{align}
From the above expression, it is easy to check that $\beta^2\sup_{\beta_0\in [U_L, U_R]}R'_{\mu_p}(\beta_0)\rightarrow 0$, as $\beta\rightarrow 0$. This verifies condition (c) and the result follows.
\end{proof}

The integral in (\ref{rom1}) has a closed form expression, which can be easily computed. We refrain from writing this explicitly for notational clarity. However, for $p=1/2$, in which case $m=0$, (\ref{rom1}) simplifies to the expression in Corollary \ref{cor:rom}.

\begin{remark} Marinari et al. \cite{marinari_parisi} predicted that replica symmetry is broken in ROM with $p=1/2$ for $\beta\geq 3.84$. The exact location of symmetry breaking is, however, unclear. Corollary \ref{cor:rom} shows that there exists a $\beta_{0}$ up to which the limit of free energy is given by the annealed limit. The value of $\beta_{0}$ can be calculated as follows: Let $F(x, y)=\beta(x+y)+\log\cosh\beta(x+y)$, and 
\begin{equation}
(x^*(\beta), y^*(\beta)):=\arg\sup_{x, y\in \reals} (F(x,y) - T_{\mu}(x) - T _{\mu}(y)),
\label{betacritical}
\end{equation}
where $T_{\mu}(z):=-\frac{1}{4}\log(1-z^2)$, for $z\in [-1, 1]$. It is follows from the proof of Theorem \ref{main} (see (\ref{2momentvar})) that $\beta_0$ is largest $\beta\geq 0$ such that the $x^*(\beta)=y^*(\beta)$. Numerically solving the optimization problem (\ref{betacritical}) approximately gives $\beta_0 \leq 2.7$, proving that replica symmetry is preserved for $\beta\leq 2.7$.
\end{remark}

\subsection{Gaussian Hopfield Model} In the Gaussian Hopfield model the coupling matrix $J= \frac{1}{p} XX^T$, where $X= ((X_{ij}))$ is a $N \times p$  matrix with i.i.d. $\cN(0, 1)$. For simplicity, we assume $0< c_1 < N/p < c_2 <1$. In this case,  spectral distribution of $J$ converges weakly almost surely to the Marchenko-Pastur law with density 
\begin{align}
f(x) = \frac{\sqrt{4\lambda - (x -1 - \lambda)^2}}{ 2 \pi x}, \quad  x \in ( (1- \sqrt{\lambda})^2, (1+ \sqrt{\lambda})^2),
\label{mp}
\end{align}
where $p/N \rightarrow \lambda $. Thus, $\lambda_{\min} = (1- \sqrt{\lambda})^2$ and $\lambda_{\max} = (1 + \sqrt{\lambda})^2$ in this example. 

Using the above density and Theorem~\ref{main} the limit of the free energy can be derived for high temperatures. 
 
\begin{ppn}In the Gaussian Hopfield model, for $\beta$ sufficiently small,
\begin{align}
\frac{1}{N} \log Z_N(\beta) \pto   I_f(\beta)=\frac{\lambda}{2} \log \left( \frac{1}{1- 2 \beta}\right). 
\label{mplimit}
\end{align}
\end{ppn}

\begin{proof} The Hilbert Transform of the Marchenko-Pastur law (\cite[Example 3.3.5]{hiai_petz})  is known to be 
\begin{align}
H_{f}(x) = \frac{ x + 1 - \lambda - \sqrt{(x-1- \lambda)^2 - 4\lambda}}{2x}. 
\end{align}
Thus, in this example, $H_{\max} = 1/(1+ \sqrt{\lambda})$ and $H_{\min}= 1/(1-\sqrt{\lambda})$, which implies that $x_{\max}= \lambda + \sqrt{\lambda}$ and $x_{\min} = \lambda- \sqrt{\lambda}$. 
The $R$-transform~(\ref{R_transform}) is $R_f(z)= \lambda/(1-z)$. Hence, $I_f$~(\ref{Ibeta}) can be computed easily, which gives the formula in~(\ref{mplimit}). Finally, to check condition (c) in Theorem 1 note that
\begin{align*}
R'(z) = \frac{\lambda}{(1- z)^2}. 
\end{align*}
The above representation implies that  $\zeta(\beta) \to 0$ as $\beta \to 0$, thus verifying the required condition.

The result now follows if the spectrum of the coupling matrix $J$ satisfies Hypothesis \ref{hyp:dist}. To this end, note that simple modifications of the arguments in \cite[Corollary 2]{dallaporta2} yield the following: there exists a constant $c>0$ such that 
\begin{align}
\label{eq:hopfield_wasserstein}
\E(d_2(\mu_N(D), \E(\mu_N(D))) ) \leq \frac{(\log N)^{c \log \log N}}{N}. 
\end{align}
Hypothesis~\ref{hyp:dist} follows by an application of Markov's inequality. 
%
%
\end{proof}


\section{Proofs} 

\subsection{Proof of Proposition \ref{ppn:first_step}}
\label{pfppn}

In this section the proof of Proposition \ref{ppn:first_step} is presented. Fix $\delta>0$ and recall that 
$\Phi_N(\beta, O, D)= \frac{1}{N}\log Z_N(\beta, O, D)$. Therefore, by triangle inequality, 
\begin{align}
\P(| \Phi_N(\beta) - \Gamma_N(\beta,  \E(D) ) | > \delta) &\leq T_1+ T_2, 
\label{trieq}
\end{align}
where
\begin{align}
T_1 &= \P\left( \left| \frac{1}{N} \log Z_N(\beta, O ,D) - \frac{1}{N} \E_{0} \log Z_N(\beta, O , D) \right| > \frac{\delta}{2} \right), 
\label{T1}
\end{align}
and
\begin{align}
T_2 &= \P\left( \left| \frac{1}{N} \E_0 ( \log Z_N(\beta, O , D) ) - \frac{1}{N} \E_0 \log Z_N(\beta , O , \E(D))  \right | > \frac{\delta}{2} \right).
\end{align}

We first control $T_2$. By the rotational invariance of $O(N)$,  $ \frac{1}{N} \E_0 \log Z_N(\beta, O , D)$ is actually a function of only the empirical distribution $\mu_N(D):=\frac{1}{N}\sum_{i=1}^{N} \delta_{d_i}$, where $D= \diag(d_1, \cdots, d_N)$. Thus, without loss of generality assume $d_1 \geq d_2 \geq \cdots \geq d_N$. Let $O=[o_1: o_2: \cdots: o_N]$ be the columns of the matrix $O$. By the Cauchy-Schwarz inequality, 
\begin{equation}
|\us^{T} ODO^{T} \us-\us^{T} O\E(D)O^{T} \us|=\left|\sum_{i=1}^N (d_i-\E(d_i)) (\vec \sigma^{T} o_i)^2\right|\leq N \sqrt{ \sum_i (d_i - \E(d_i))^2},
\label{w2d} 
\end{equation}
since $(\vec \sigma^{T} o_i)^2=N$, for all $i\in [N]$. This implies that
\begin{align}
\left|\frac{1}{N}\E_0\left(\log \frac{Z_N(\beta, O, D)}{Z_N(\beta, O, \E(D))} \right) \right| \leq & \beta \sqrt{ \sum_i (d_i - \E(d_i))^2}\nonumber\\
=& \beta \sqrt N W_2 (\mu_N(D) ,\mu_N(\E (D) )),
\end{align}
where the last step uses  $W_2^2(\mu_N(D) ,\mu_N(\E(D))) =\frac{1}{N} \sum_{i=1}^{N} (d_i - \E(d_i))^2$ .  Therefore,
\begin{align}
T_2 \leq \P\left( W_2(\mu_N(D), \E(\mu_N(D))) > \frac{\delta}{2\beta \sqrt{N}}\right ) \to 0  
\label{T2l}
\end{align}
by Hypothesis \ref{hyp:dist}, as $N\rightarrow \infty$. 


It remains to control the first term $T_1$. For $O \in O(N)$ and any fixed diagonal matrix $\Lambda= \diag(\lambda_1 ,\cdots, \lambda_N)$ define,
\begin{align}
F_{\Lambda}(O)= \frac{1}{N} \log \sum_{\us \in S_N} \exp(\beta \us^{T} O\Lambda O^{T} \us). 
\label{fo}
\end{align}
Let $\| \Lambda \|_{\infty} = \max_{1\leq i \leq N} |\lambda_i|$. Moreover, for any $N\times N$ symmertic matrix $A$, denote the spectral norm by $||A||_{2}=\sup_{\vec x\in \reals}\frac{||A\vec x||_2}{||\vec x||_2}$  and the Frobenius norm by $||A||_F=(\tr(A^2))^{\frac{1}{2}}$. It is easy to see that for $O_1, O_2\in O(N)$ and a unit vector $\vec x$ (that is $\|\vec x\|_2=1$),
\begin{align}
|\vec x^{T}O_1 \Lambda O_1^{T}\vec x - \vec x^{T}O_2 \Lambda O_2^{T}\vec x| &\leq  |\vec x^{T} O_1 \Lambda (O_1- O_2)^{T} \vec x| + | \vec x^{T} (O_1 - O_2) \Lambda  O_2^{T}\vec x| \nonumber \\ 
&\leq 2 \|\Lambda \|_{\infty} \|O_1- O_2\|_{2} \nonumber \\
&\leq 2 \| \Lambda \|_{\infty} \|O_1- O_2\|_{F}. \label{concentration_bound} 
\end{align}
Thus, using \eqref{concentration_bound},
\begin{align}
|F(O_1) - F(O_2)| &= \frac{1}{N} \left| \log \frac{Z_N(\beta, O_1, \Lambda)}{Z_N(\beta, O_2, \Lambda)} \right|  \leq 2 \| \Lambda \|_{\infty} \beta \|O_1 - O_2\|_{F}. \nonumber 
\end{align}
This implies $F$ is Lipschitz with respect to the Frobenius norm. 

Sub-gaussian tail inequalities are known for Lipschitz functions on $SO(N)$ (see Gromov and Milman \cite{gromov_milman}). This can be used to complete the proof as follows: 
Now, let $T$ be the operator which takes $O\in SO(N)$ and changes the sign of the first column of $O$.  Clearly, for $O \in SO(N)$, $F(O)=F(TO)$. Let $\P_1$ and $\E_1$ be Haar measure and the expectation with respect it on $SO(N)$, respectively. Thus, $\E_0(F_D(O))= \E_1(F_D(O))$, and recalling (\ref{T1}) and (\ref{fo}) it follows that 
\begin{align}
T_1\leq &\E \P_1\left (|F_D(O)-\E_1 (F_D(O))|>\frac{\delta}{2},  \| D \|_{\infty} \leq  M_N \right)+ \P( \| D \|_{\infty} > M_N) \nonumber\\
\leq &\exp \Big (-\frac{C N \delta^2}{ \beta^2 M_N^2} \Big) + \P( \| D \|_{\infty} > M_N),
\label{T1l}
\end{align}
where $C>0$ is a universal constant. By Hypothesis \ref{hyp:dist}, the RHS above goes to zero as $N \to \infty$. 

Combining (\ref{T2l}) and (\ref{T1l}) with (\ref{trieq}) the result follows. 

%


\subsection{Proof of Theorem \ref{main}}
\label{pfmain} 

By concentration arguments identical to those used in controlling the term $T_1$ in Proposition \ref{ppn:first_step}, the following lemma can be proved. 

\begin{lemma}
\label{concentration}
For any $\beta>0$, there exists an universal constant $c$, independent of $N$, such that
\begin{align}
\mathbb P\left( \left| \Phi_N(\beta, O, \E(D))  -  \E_0 \Phi_N(\beta, O, \E(D)) \right|  > \delta \right) \leq \exp{(- cN \delta^2/ \beta^2)} . 
\end{align} 
\end{lemma}
\noindent

The  proof  of Theorem \ref{main} also requires computing the first and second annealed moments of $Z_N(\beta, O, \E(D))$. 

\begin{ppn}
\label{second_moment}
Under the assumptions of Theorem \ref{main}, for $\beta$ sufficiently small (possibly depending on the limiting measure $\mu$), 
\begin{align}
\lim_{N \to \infty} \frac{1}{N} \log \E_0(Z_N(\beta, O ,\E(D))) = \lim_{N \to \infty} \frac{1}{2N} \log \E_0(Z_N(\beta, O ,\E(D))^2).  
\end{align}
\end{ppn}

The above lemma is the most challenging part of our argument and the proof is deferred to Section \ref{section:second_moment}.  

The proof of Theorem \ref{main} can be completed easily by combining Lemma \ref{concentration} and Proposition \ref{second_moment} with Theorem \ref{thm:annealed}. To this end,  set
$\gamma_0 = \frac{4 \E_0 Z_N(\beta, O, \E(D))^2}{ (\E_0 Z_N(\beta, O , \E(D)))^2}$. Recall the definition of the {\it{annealed free energy} }
\begin{align*}
\phi_N(\beta, \Lambda)=\frac{1}{N}\log \E_0 Z_N(\beta, O, \Lambda). 
\end{align*}
Then by \cite[Lemma 4.1.1]{montanarinotes}
\begin{align}
\P\left( \left | \Phi_N(\beta, O, \E(D))  -  \phi_N(\beta, \E(D)) \right| < \frac{1}{N} \log \gamma_0 \right) \geq \frac{1}{\gamma_0}. 
\label{eq:tails}
\end{align}
Also, note that  $\Gamma_N(\beta, \E(D))= \E_0 \Phi_N(\beta, O, \E(D))$. Thus, inequality \eqref{eq:tails} combined with Lemma \ref{concentration} gives 
\begin{equation}
\lim_{N\rightarrow \infty}\Gamma_N(\beta, \E(D)) =\lim_{N\rightarrow \infty} \frac{1}{N} \log \E_0(Z_N(\beta, O ,\E(D)))=I_{\mu}(\beta),
\label{1momentgm}
\end{equation}
where the last step uses Theorem \ref{thm:annealed}. 
Finally, using Proposition \ref{second_moment}, Theorem \ref{main} follows.

\subsubsection{\textbf{Proof of Proposition \ref{second_moment}}}
\label{section:second_moment}

For any function $f: S_N\times S_N\mapsto \mathbb R$, denote by $\E_1 f(\vec \sigma, \vec{\tau})=\frac{1}{2^{2N}}\sum_{\vec \sigma, \vec{\tau} \in S_N}f(\vec \sigma, \vec{\tau})$, the expectation over the uniform measure over $S_N\times S_N$. Let $\Lambda=\E(D)=\diag(\lambda_1, \lambda_2, \cdots, \lambda_N)$. Therefore,

\begin{align}
\E_0(Z_N(\beta, O , \E(D))^2 =&\E_0(Z_N(\beta, O , \Lambda)^2)\nonumber\\
&=  \E_1 \E_0 \exp\left(N\beta\tr\left\{O \Lambda O^{T}\left(\frac{\vec \sigma\,\us^{T}}{N}+\frac{\vec{\tau}\,\vec{\tau}'}{N}\right)\right\} \right) \nonumber\\
&= \E_1 \E_0 \exp\left(N\beta \left\{\left(1+\frac{\us^{T}\vec{\tau}}{N}\right)(O \Lambda O^{T})_{11}+\left(1-\frac{\us^{T}\vec{\tau}}{N}\right)(O \Lambda O^{T})_{22} \right\} \right),\nonumber 
\end{align}
where we use the observation that the non-zero eigenvalues of $(\us \us^T + \vec{\tau} \vec{\tau}^T)/N$ are $(1+ \us ^T \vec{\tau}/N)$ and $(1- \us^T \vec{\tau}/N)$ respectively. 
Let  $V_1=(O \Lambda O^{T})_{11}$ and $V_2=(O \Lambda O^{T})_{22}$. By interchanging the order of the expectation and observing that $\E_1 e^{\lambda \us^{T}\vec{\tau}}=(\cosh \lambda)^N$, for any $\lambda\in \mathbb R$, it follows that
\begin{eqnarray}
\E_0(Z_N(\beta, O, \Lambda)^2)&=& \E_0\left(\exp\left(N F(V_1, V_2)\right)\right) 
\label{2moment}
\end{eqnarray}
where $F(x,y)= \beta (x+y) + \log \cosh \beta(x-y)$. 

The non-negativity of the $\log \cosh$ function trivially implies that $F(x,y)\geq \beta(x+y)$. Then by \cite[Theorem 1.7]{guionnet_maida}, for $\beta$ sufficiently small, we have
\begin{align}
\liminf_{N \to \infty} \frac{1}{2N} \log \E_0(Z_N(\beta, O , \Lambda)^2) & \geq  \lim_{N \to \infty} \frac{1}{2N} \log \E_0 \exp(N \beta (V_1 + V_2))\nonumber\\
&=  \lim_{N \to \infty} \frac{1}{2N} \log \E_0 \exp( N \beta (V_1 + V_2))=  I_{\mu}(\beta),
\label{eq:lowerb}
\end{align}
where $\mu$ is the limit of the empirical measure $\mu_N(\E(D)):=\frac{1}{N}\sum_{i=1}^N \delta_{\E(d_i)}$.

For the upper bound, let $\textbf{\em{X}}=(X_1, X_2, \ldots, X_N)'$ and $\textbf{\em{Y}}=(Y_1, Y_2, \ldots, Y_N)'$ be i.i.d. $\dN(0, \mathrm I)$. If $$\textbf{\em{Z}}= \textbf{\em{Y}} - \frac{\langle \textbf{\em{X}} ,\textbf{\em{Y}} \rangle}{ \langle \textbf{\em{X}}, \textbf{\em{X}} \rangle } \textbf{\em{X}}=(Z_1, Z_2, \ldots, Z_N)',$$ then 
\begin{align}
(V_1 , V_2) \overset{\sD}{=}  \left(\frac{\sum_{i=1}^N \lambda_i X_i^2}{\sum_{i=1}^N X_i^2} , \frac{\sum_{i=1}^N \lambda_i Z_i^2}{\sum_{i=1}^N Z_i^2} \right). 
\label{V1V2}
\end{align} 
Let $V_2'= \frac{\sum_{i=1}^N \lambda_i Y_i^2}{\sum_{i=1}^N Y_i^2}$. Note that $V_1$ and $V_2'$ are independent, but $V_1$ and $V_2$ are not. The following lemma shows that we can replace $V_2$ by $V_2'$ to get an upper bound:

\begin{lemma}
\label{lemma:upper_bound}
Under the assumptions of Theorem \ref{main}, for any $\beta >0$,
\begin{align}
\limsup_{N \to \infty} \frac{1}{2N} \log \E_0(Z_N(\beta, O , \Lambda)^2) \leq \lim_{ N \to \infty } \frac{1}{2N} \log \E_0 \exp( N F (V_1 ,V_2')),
\label{upperb} 
\end{align}
where $F(x,y)= \beta (x+y) + \log \cosh \beta(x-y)$.
\end{lemma}

\begin{proof}The lemma will be established using a ``localization" argument similar to the one used in \cite{guionnet_maida}. Fix $\kappa < 1/2$ and
\begin{align} 
B_N(\kappa) = \left\{ \left| \frac{1}{N} \sum_{i=1}^{N} X_i^2 -1 \right| \leq N^{-\kappa}, \left|\frac{1}{N} \sum_{i=1}^{N} Y_i^2 -1\right| \leq N^{-\kappa}, \left|\frac{1}{N} \sum_{i=1}^{N} X_i Y_i \right| \leq N^{-\kappa} \right\} . 
\end{align}
We adopt the following system of coordinates in $\reals^{2N}$: $r, \alpha_1^{(1)} ,\cdots, \alpha_{N-1}^{(1)}$ are the polar coordinates of $\textbf{\em{X}}$, $r_2= || \textbf{\em{Y}}||$, $\beta_2$ is the angle between $\textbf{\em{X}}$ and $\textbf{\em{Y}}$, and $\alpha_1^{(2)},\cdots, \alpha_{N-2}^{(2)}$ are the angles needed to spot $\textbf{\em{Y}}$ on a cone of angle $\beta_2$ around $\textbf{\em{X}}$. It is easy to see that $(V_1,V_2)$ is a function of the $\alpha$'s while the event $B_N(\kappa)$ is determined by $r$ and the $\beta$'s. So $(V_1,V_2)$ and $B_N(\kappa)$ are independent. 

Let $I_N = \E_0\exp(N F(V_1,V_2))$. By (\ref{2moment}), $\frac{1}{2N} \log \E_0(Z_N(\beta, O , \Lambda)^2) = \frac{1}{2N}\log I_N$. Therefore, to prove (\ref{upperb}) it suffices to show that 
\begin{align}
I_N \leq \ve(N,\kappa) \E_0( \mathbf{1}_{B_N(\kappa)} \exp(N F(V_1,V_2')) )  
\label{eq:sandwich}
\end{align}
where $\ve(N,\kappa) \leq C(\kappa) \exp(N^{1-2\kappa})$ for some constant $C(\kappa)$ and $N$ sufficiently large. 

By bounding the moment generating functions of $X_1^2$ and $X_1Y_1$ suitably in a neighborhood of zero, we get
\begin{align}
& \P(B_N(\kappa)^c) \nonumber\\
&\leq \P\left( \Big|\frac{1}{N} \sum_{i=1}^N X_i^2 - 1 \Big| \geq N^{-\kappa} \right) + \P\left( \Big|\frac{1}{N} \sum_{i=1}^N Y_i^2 - 1\Big| \geq N^{-\kappa} \right) + \P\left( \Big| \frac{1}{N} \sum_{i=1}^N X_i Y_i \Big| > N^{-\kappa}\right)\nonumber \\
&\leq  C'(\kappa) \exp{(-c N^{1- 2 \kappa})}, 
\label{eq:mgf} 
\end{align}
for some constants $C'(\kappa)$, $c>0$ and $N$ sufficiently large. Now, using the independence of $(V_1,V_2)$ and $B_N(\kappa)$, 
\begin{align}
I_N \leq \frac{1}{\P(B_N(\kappa))} \E_0( \mathbf{1}_{B_N(\kappa)} \exp(N F(V_1,V_2)) )\leq  \ve(N,\kappa)\E_0( \mathbf{1}_{B_N(\kappa)} \exp(N F(V_1,V_2)) ). 
\end{align}
By the Lipschitz property of the $\log \cosh$ function $|F(x,y) - F(x,z)| \leq 2\beta |y-z|$. Therefore,
\begin{align}
I_N \leq \ve(N,\kappa) \E_0( \mathbf{1}_{B_N(\kappa)} \exp( N F(V_1, V_2') + 2N \beta |V_2 - V_2'|) ). 
\end{align}
The upper bound in (\ref{eq:sandwich}) follows if, on the set $B_N(\kappa)$, $|V_2 - V_2'| \lesssim N^{-\kappa}$. To this end, note that on $B_N(\kappa)$, $\frac{1}{N} ||\textbf{\em{Y}} -\textbf{\em{Z}}||^2 \lesssim N^{-\kappa}$. Further, on $B_N(\kappa)$,
\begin{align}
\frac{1}{N} |\textbf{\em{Z}}^{T} \Lambda \textbf{\em{Z}} - \textbf{\em{Y}}^{T} \Lambda \textbf{\em{Y}} | &\leq \frac{2}{N} \|\Lambda\|_\infty  \|\textbf{\em{Z}}\| \|\textbf{\em{Z}} - \textbf{\em{Y}} \| \lesssim N^{-\kappa},
\end{align}
since $\|\Lambda\|_\infty$ is finite by Hypothesis \ref{hyp:exp_dist}(b).

From this it is easy to see that on the set $B_N(\kappa)$, $|V_2 - V_2'| \lesssim N^{-\kappa}$, and the proof is complete.
\end{proof}

\begin{lemma}
\label{lemma:variational_problem}
Under the assumptions of Theorem \ref{main}, for $\beta\geq  0$ sufficiently small,
\begin{align}
\lim_{ N \to \infty} \frac{1}{2N} \log \E_0 \exp( N F(V_1, V_2')) = I_\mu(\beta). 
\end{align}
\end{lemma}

The proof of the above lemma is given below in Section \ref{variational_problem}. 
Note that the Lemma \ref{lemma:variational_problem} together with (\ref{eq:lowerb}) and (\ref{upperb}) gives 
\begin{eqnarray}
\lim_{N \to \infty} \frac{1}{2N} \log \E_0(Z_N(\beta, O ,\E(D))^2)=\lim_{N \to \infty} \frac{1}{N} \log \E_0(Z_N(\beta, O ,\E(D)))=I_\mu(\beta),
\end{eqnarray}
where the last equality uses Theorem \ref{thm:annealed}. This completes the proof of Proposition \ref{second_moment}. \\

\noindent{\it\textbf{Proof of Lemma \ref{lemma:variational_problem}:}} 
\label{variational_problem}
The proof of this lemma follows from a large deviation result established in \cite{guionnet_maida}. Recall the Hilbert transform and the $R$-transform of a probability measure $\nu$ defined in \eqref{hilbert_transform}, and \eqref{R_transform}, respectively. Denote the inverse of $H_{\nu}$ by $K_{\nu}$, and that of $R_{\nu}$ by $Q_{\nu}$. Refer to \cite{guionnet_maida} for further details about the Hilbert and the $R$-transforms.

Also, recall \eqref{eq:hmax} 
\begin{align}
x_{\max} = \lambda_{\max} - \frac{1}{ H_{\max}} \,\,\, \textrm{and} \,\,\, x_{\min} = \lambda_{\min} - \frac{1}{H_{\min}}, 
\label{xmaxmin}
\end{align}
where $H_{\max} = \lim_{ z \downarrow \lambda_{\max}} H_{\nu}(z)$ \textrm{and} $H_{\min} = \lim_{z \uparrow \lambda_{\min}} H_{\nu}(z)$.  Finally, for $\kappa \in (\lambda_{\min} , \lambda_{\max})^c$, define
\begin{align}
h_x{(\kappa)} = \int \log \frac {\kappa - \lambda}{\kappa - x} \de \nu(\lambda), 
\label{hx}
\end{align}
and $h_x^{\min} = \lim_{\kappa \uparrow \lambda_{\min}} h_x(\kappa)$ and $h_x^{\max} = \lim_{\kappa \downarrow \lambda_{\max}} h_x(\kappa)$. 

The following proposition, proved in \cite{guionnet_maida}, gives the large deviations rate function for the random variable $V_1$. 

\begin{propo}(\cite[Proposition 5.1]{guionnet_maida})
\label{prop:ldp}
If the sequence of non-random empirical measures $\mu_N(\Lambda)=\frac{1}{N}\sum_{i=1}^N \delta_{\lambda_i}\to \mu$ and satisfies Hypothesis \ref{hyp:exp_dist}, then the law of the random variables $V_1$ defined in (\ref{V1V2}) satisfies a large deviation principle with scale $N$ and good rate function 
\begin{align}\label{Tcases}
T_\mu(x) = 
\begin{cases}
\frac{1}{2} h_x(K_{\mu}(Q_{\mu}(x))) & \textrm{if   } \, x \in [x_{\min} , x_{\max}], \\
\frac{1}{2} h_x^{\max} & \textrm{if   } \, x \in ] x_{\max} , \lambda_{\max} [, \\
\frac{1}{2} h_x^{\min} & \textrm{if   } \, x \in ]\lambda_{\min} , x_{\min} [, \\
\infty & \textrm{  otherwise}. 
\end{cases} 
\end{align}
\end{propo}

Since the  empirical measures $\mu_N(\Lambda)$ satisfies Hypothesis \ref{hyp:exp_dist}, by Varadhan's lemma \cite{dembo_zeitouni} we get 
\begin{align}
\lim_{N \to \infty} \frac{1}{2N} \log \E_0\exp( N F(V_1 ,V_2')) = \frac{1}{2}\sup_{x, y\in \R} (F(x,y) - T_{\mu}(x) -T _{\mu}(y)), 
\label{2momentvar}
\end{align}
where $T_{\mu}(\cdot)$ is the good rate function of $V_1$ and $F(x, y)=\beta(x+y)+\log\cosh\beta(x-y)$. 

Set $\psi(x,y) = F(x,y) - T_{\mu}(x) - T_{\mu}(y)$. To prove Lemma~\ref{lemma:variational_problem} it suffices to establish that for $\beta$ sufficiently small, 
\begin{align}\label{imu}
\sup_{x, y\in \R} \psi(x,y) = \sup_{(x,y) \in [x_{\min}, x_{\max}]^2} \psi(x,y) = 2 I_{\mu}(\beta),
\end{align}
where $x_{\max}$ and $x_{\min}$ are as defined in~\eqref{xmaxmin}.

To this end note that, 
\begin{align}
\label{eq:first_order1}
\frac{\partial \psi}{\partial x} &= \beta + \beta \tanh \beta(x-y) - T_{\mu}'(x) , \\
\frac{\partial \psi}{\partial y} &= \beta - \beta \tanh \beta(x-y) - T_{\mu}'(y). \label{eq:first_order2}
\end{align}
We begin by showing that the maxima of $\psi$ on the set $[\lambda_{\min}, \lambda_{\max}]^2 \backslash [x_{\min}, x_{\max}]^2$ is attained on the boundary of the set $[x_{\min}, x_{\max}]^2$, for $\beta < H_{\max}/4$: For instance, for $(x,y) \in ( x_{\max} , \lambda_{\max}] \times [x_{\min}, x_{\max}]$ (as in Figure~\ref{fig:range}), by~\eqref{Tcases} $T_{\mu}'(x)=\frac{1}{2}\frac{\de}{\de x}h_x^{\max} = \frac{1}{2(\lambda_{\max} - x)}$. Then for $\beta < H_{\max}/4$, using $\tanh \beta (x- y)\leq 1$ and $x\geq x_{\max}$,
\begin{align}\label{neg}
\frac{\partial \psi }{\partial x}(x,y) = \beta (1 + \tanh \beta ( x- y) ) - \frac{1}{2(\lambda_{\max}- x)} <0. 
\end{align} 
Therefore, $\psi(x, y)\leq \psi(x_{\max}, y)$ for $(x,y) \in ( x_{\max} , \lambda_{\max}] \times [x_{\min}, x_{\max}]$. Other points in the region $[\lambda_{\min}, \lambda_{\max}]^2 \backslash [x_{\min}, x_{\max}]^2$ can be dealt with similarly.

\begin{figure*}[h]
\centering
\begin{minipage}[c]{1.0\textwidth}
\centering
\includegraphics[width=2.5in]
    {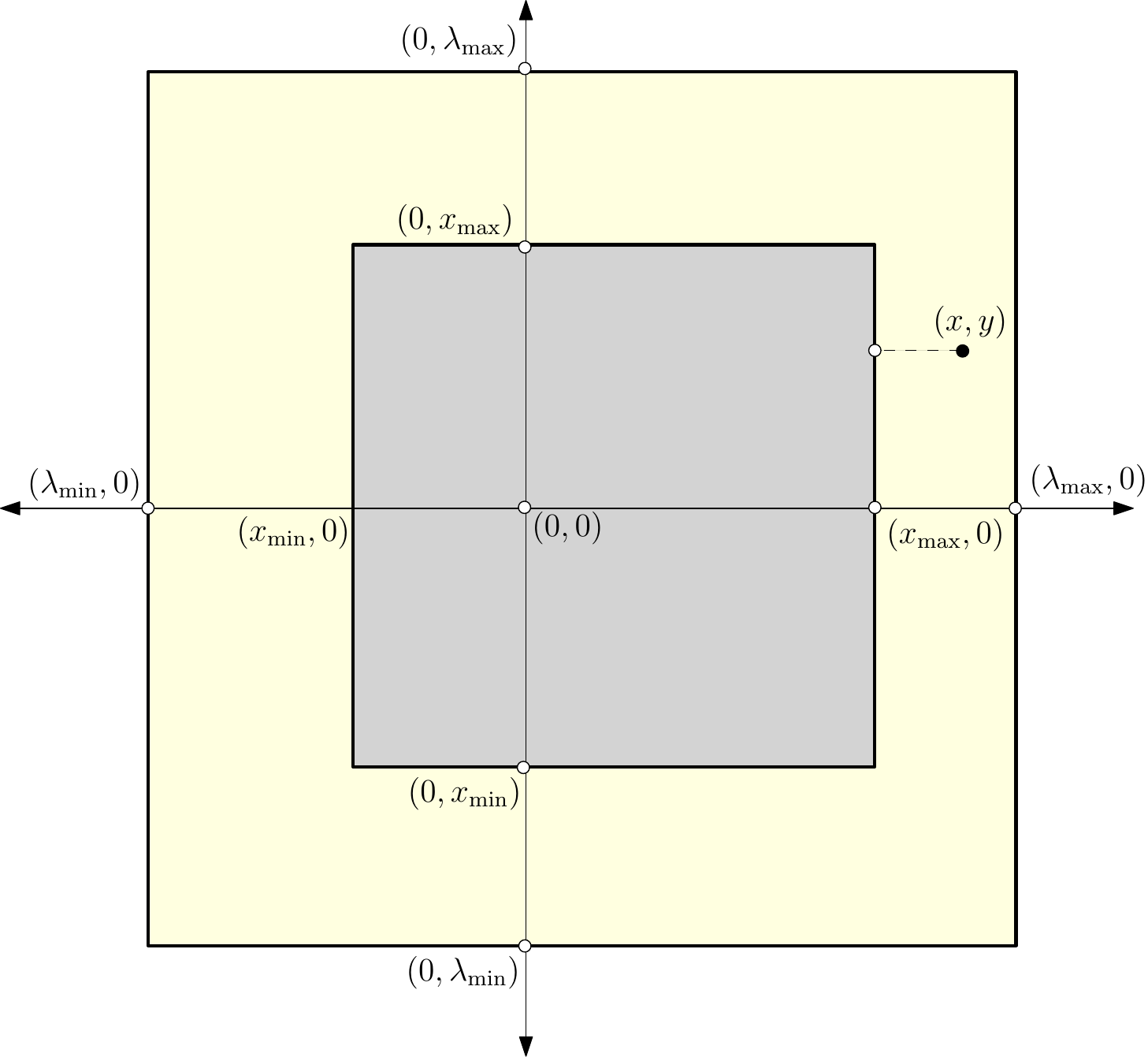}\\
\end{minipage}
\caption{\small{The domain of the function $\psi$.}}
\label{fig:range}
\end{figure*}

Next, we establish that for $\beta$ small enough, the maxima of $\psi$ on $[x_{\min}, x_{\max}]^2$ cannot be attained on the boundary of $[x_{\min}, x_{\max}]^2$: For $y \in (x_{\min}, x_{\max})$, 
$$\frac{\partial}{\partial x}\psi(x_{\max}- , y)=\beta(1+\tanh \beta(x_{\max}-y))-\frac{1}{2}H_{\max} < 0.$$
if $\beta < 4 H_{\max}$ as in~\eqref{neg}. This ensures that the maxima cannot be attained for $x= x_{\max}$. The same analysis implies that the maxima is not attained for $y = x_{\max}$. This establishes the required assertion.

It remains to analyze the function $\psi$ on $(x_{\min}, x_{\max})^2$. Note that for $x\in (x_{\min}, x_{\max})$, $T'_\mu(x)=\frac{1}{2}Q_{\mu}(x)$. Therefore, from~\eqref{eq:first_order1}-\eqref{eq:first_order2} any stationary point of $\psi$ in $(x_{\min}, x_{\max})^2$ is a solution of the system of equations
\begin{align}
x &= R_{\mu}(2 \beta ( 1 + \tanh \beta(x-y))) \label{eq:fixedpoint1},\\
y &= R_{\mu}(2 \beta (1 - \tanh \beta(x-y))), \label{eq:fixedpoint2}
\end{align}
since $R_\mu$ is the inverse of $Q_\mu$.

Let $a(\beta)$ be the solution of $\beta = T_{\mu}' (\cdot)=\frac{1}{2}Q_{\mu}(\cdot)$. 
It is easy to verify that $x^*(\beta)=y^*(\beta)=a(\beta)$ is a critical point of $\psi$ in $(x_{\min}, x_{\max})^2$.  It remains to show that for $\beta$ sufficiently small the maximum in (\ref{2momentvar}) is attained at $x^*(\beta)=y^*(\beta)=a(\beta)$. This implies Lemma~\ref{lemma:variational_problem}, since by \cite[Lemma 5.7]{guionnet_maida}, $\{ \beta a(\beta)-T_{\mu}(a(\beta))\}=I_{\mu}(\beta)$.

To show that the maximum in \eqref{2momentvar} is attained at $x^*(\beta) = y^*(\beta) = a(\beta)$ we will show (in Lemma~\ref{lemma:fixedpt} below) that for $\beta$ sufficiently small, the system of equations \eqref{eq:fixedpoint1}-\eqref{eq:fixedpoint2} has a unique solution $(a(\beta), a(\beta))$ which is a local maxima. This establishes that $(a(\beta), a(\beta))$ maximizes $\psi$ in $(x_{\min}, x_{\max})^2$ and~\eqref{imu} follows.

\begin{lemma}
\label{lemma:fixedpt}
For $\beta$ sufficiently small and $R_{\mu}$ satisfying condition (c) in Theorem \ref{main}, the system of equations \eqref{eq:fixedpoint1}-\eqref{eq:fixedpoint2} has a unique solution which is a local maximum. 
\end{lemma}

\begin{proof}
For $(x, y)\in [x_{\min}, x_{\max}]^2$, define 
\begin{align}
G(x,y) = 
\left( \begin{array}{c}
R_{\mu}(2\beta ( 1 + \tanh \beta(x- y))) \\
R_{\mu}( 2 \beta ( 1 - \tanh \beta(x-y)))
\end{array} \right). 
\end{align}
For $\beta < H_{\max}/ 4$, $G$ maps $[x_{\min}, x_{\max}]^2$ to $[x_{\min}, x_{\max}]^2$. Consider the set $M = [x_{\min}, x_{\max}]^2$ equipped with the $L_1$ metric, that is, $d_1 (( x_1, y_1), (x_2, y_2)) = | x_1 - x_2| + | y_1 - y_2|$. We note that $(M,d_1)$ is a complete metric space. Now, by the mean value theorem and using $\sech^2 \leq 1$, 
\begin{align*}
| R_{\mu} ( 2 \beta ( 1 + \tanh \beta(x_1 - y_1) )) - R_{\mu}( 2 \beta ( 1 + \tanh \beta ( x_2 - y_2))) | \leq 2 \beta^2 \zeta(\beta) d_1((x_1, y_1), (x_2, y_2)). 
\end{align*}
where $\zeta(\beta)=\beta^2 \sup_{\beta_0 \in [ U_L, U_R] } | R'(\beta_0)|$ and $U_L$ and $U_R$ are as in~\eqref{interval}. This implies, by condition (c) of Theorem \ref{main}, that $G$ is a contraction. Thus, using the  Banach Fixed Point Theorem \cite{simmons}, $G$ has a unique fixed point. This implies $(a(\beta), a(\beta))$ is the unique fixed point of $G$, and the only  stationary point of $\psi$ in $(x_{\min}, x_{\max})^2$.

Finally, we establish that the critical point $(a(\beta), a(\beta))$ is a local maxima of the function $G$. To see this note that 
the Hessian 
\begin{align}
\grad^2 \psi(x,y) = -\diag(T_{\mu}''(x) , T_{\mu}''(y)) + \beta^2 \sech^2 \beta(x-y) 
\left(\begin{matrix}
1 & -1 \\
-1 & 1
\end{matrix} \right).
\label{2derivative} 
\end{align}
Using $T_{\mu}''(x) = \frac{1}{2 R_{\mu}'(Q(x))}$, we note that $\grad^2 \psi(a(\beta) , a(\beta))$ is negative definite provided
$\beta^2 R'(2\beta) \leq 1/4$. This condition is satisfied for $\beta$ sufficiently small as condition (c) holds, thus completing the proof. 
\end{proof}



\begin{remark}
\label{rem:sk}
In the SK model, $R_\rho(z)=z$ (defining $R_{\rho}(0)=0$), $x_{\max}=1$, $x_{\min}=-1$, and $\rho$ is the semi-circle law (\ref{semicircle}). Moreover, 
\begin{align}
T_{\rho}(x) = \begin{cases}
\frac{1}{4} - \log (2+x) &  x \in [ -2, -1],\\
\frac{x^2}{4} & x \in [-1, 1], \\ 
 \frac{1}{4} - \log(2-x)& x \in [1,2]. 
\end{cases}
\end{align}
Note that $T_{\rho}(x) > x^2/4$ in $[-2,-1] \cup [1,2]$. Therefore, $\psi(x, y)\leq F(x, y)-\frac{x^2}{4}-\frac{y^2}{4}:=\tilde \psi(x, y)$ for $(x, y) \in [-2, 2]^2$. It is easy to see that $(2\beta, 2\beta)\in [-1, 1]^2$, for $\beta<1/2$, is a stationary point of $\tilde \psi$. The Hessian $\grad^2 \tilde \psi(x,y) $ is negative definitive for $(x, y)\in [-2, 2]^2$ if 
\begin{enumerate}[(a)]

\item  $(\grad^2 \tilde \psi(x,y))_{11}=\frac{1}{2}+\beta^2\sech^2\beta(x-y)< 0$, which holds whenever $\beta< \frac{1}{\sqrt 2}$; and 

\item $\det(\grad^2 \tilde \psi(x,y))=\frac{1}{4}+\beta^2\sech^2\beta(x-y)<0$, whenever $\beta< \frac{1}{2}$.
\end{enumerate}
Therefore, the maximum of $\tilde \psi$ is attained at $(2\beta, 2\beta)\in [-1, 1]^2$, for $\beta<1/2$, where the two functions $\tilde \psi$ and $\psi$ agree. Thus, for $\beta<1/2$, which is the entire replica symmetric phase,  the limit in Corollary \ref{sk} holds.
\end{remark}

\noindent \textbf{Acknowledgements} The authors thank Amir Dembo, Andrea Montanari and Sourav Chatterjee for helpful discussions. S.S. thanks Zhou Fan for help with results about random matrices.

%
\bibliographystyle{amsalpha}

\begin{thebibliography}{99}
\bibitem{alr}
M. Aizenman, J.L. Lebowitz and D. Ruelle, Some rigorous results on the Sherrington- Kirkpatrick spin glass model , {\it Comm. Math. Phys.},  Vol. 112 (1) , 3 -- 20, 1987. 

\bibitem{agz}
G.W. Anderson, A. Guionnet and O. Zeitouni, {\it An introduction to random matrices}, Vol. 118, Cambridge University Press, 2010.

\bibitem{bernasconi}
J. Bernasconi, Low autocorrelation binary sequences: statistical mechanics and configuration space analysis, {\it J. Physique}, Vol. 48, 559, 1987.

\bibitem{bovier}
A. Bovier, ACD van Enter and B. Niederhauser. Stochastic symmetry-breaking in a Gaussian Hopfield model, {\it Journal of Statistical Physics} Vol. 95 (1-2), 181--213, 1999.

\bibitem{carmona_hu}
P. Carmona  and Y. Hu, Universality in Sherrington-Kirkpatrick's spin glass model, {\it Annales de l'Institut Henri Poincare (B)}, Vol. 42 (2),  215--222, 2006.

\bibitem{hchandra}
H. Chandra, Differential operators on a semisimple Lie algebra, {\it Amer. J. Math.}, Vol. 79, 87--120. 1957.

\bibitem{chatterjee}
S. Chatterjee, A simple invariance theorem, {\tt arXiv:math/0508213}, 2005. 

\bibitem{cherrier}
R. Cherrier, D. S. Dean, A. Lef\`{e}vre, The role of the interaction matrix in mean-field spin glasses, {\it Phys. Rev. E}, Vol. 67, 046112, 2003.

\bibitem{sniady}
B. Collins, and P. \'{S}niady, New scaling of Itzykson-Zuber integrals, {\it Annales de l'Institut Henri Poincare (B)}, Vol. 43 (2), 139--146. 2007.

\bibitem{dallaporta1}
S. Dallaporta, Eigenvalue variance bounds for Wigner and covariance random matrices, {\it Random Matrices: Theory and Applications}, Vol. 1 (3), 1250007, 2012. 

\bibitem{dallaporta2}
S. Dallaporta,Eigenvalue variance bounds for covariance matrices, arXiv:1309.6265, 2013.

\bibitem{energy_landscape}
M. Degli Espost, C. Giardin\'a, and S. Graffi, Energy landscape statistics of the random orthogonal model, {\it J. Phys. A: Math. Gen.}, Vol. 36, 2983--2994, 2003.

\bibitem{dembo_zeitouni}
A.Dembo and O. Zeitouni, {\it Large deviations techniques and applications}, Second Ed., Vol. 38, Applications of Mathematics. Springer-Verlag, New York, 1998.

\bibitem{gamboa_rouault}
F. Gamboa  and A. Rouault, Canonical moments and random spectral measures, {\it Journal of Theoretical Probability}, Vol. 23 (4), 1015--1038, 2010.

\bibitem{gromov_milman}
M. Gromov and V. D. Milman, A topological application of the isoperimetric inequality, {\it Amer. J. Math.},  Vol. 105 (4), 843--854, 1983.

\bibitem{guionnet_maida}
A. Guionnet, M. Maida, Fourier view on the $R$-transform and related asymptotics of spherical integrals, {\it Journal of Functional Analysis}, Vol. 222, 435--490, 2005.

\bibitem{guionnet_zeitouni}
A. Guionnet, O. Zeitouni, Large Deviations Asymptotics for Spherical Integrals, {\it Journal of Functional Analysis}, Vol. 188(2), 461- 515, 2002. 

\bibitem{hiai_petz}
F. Hiai, D. Petz, {\it The Semicircle Law, Free Random Variables and Entropy}, American Mathematical Society, 2000. 

\bibitem{marinari_parisi_autocorrelation}
E. Marinari, G. Parisi and F. Ritort, Replica field theory for deterministic models: binary sequences with low autocorrelation, {\it J. Phys. A: Math. Gen.}, Vol. 27, 7615, 1994.

\bibitem{marinari_parisi}
E. Marinari, G, Parisi, F. Ritort, Replica field theory for deterministic models. II. A non-random spin glass with glassy behavior,  {\it J. Phys. A: Math. Gen.} Vol. 27, 7647, 1994.

\bibitem{montanarinotes}
A. Montanari, Statistical mechanics and algorithms on sparse and random graphs, St. Flour School of Probability, 2013. {\tt http://web.stanford.edu/~montanar/OTHER/STATMECH/stflour.pdf}

\bibitem{panchenko}
D. Panchenko, {\it The Sherrington-Kirkpatrick model}, Springer Science \& Business Media, 2013.


\bibitem{simmons}
G. F. Simmons, {\it Introduction to Topology and Modern Analysis}, R.E. Krieger Pub. Co., 1983.

\bibitem{talagrand}
M. Talagrand, {\it Spin Glasses, A Challenge for Mathematicians}, Springer, 2003.

\end{thebibliography}

\end{document}